\documentclass{amsart}

\usepackage{amsmath}
\usepackage{paralist}
\usepackage{graphics} 
\usepackage{epsfig} 
\usepackage{graphicx}  
\usepackage{epstopdf}
\usepackage[colorlinks=true]{hyperref}

\usepackage[labelfont=bf]{caption}
\usepackage[dvipsnames]{xcolor}

\allowdisplaybreaks

\hypersetup{urlcolor=blue, citecolor=red}

\newtheorem{theorem}{Theorem}[section]
\newtheorem{corollary}{Corollary}
\newtheorem{lemma}[theorem]{Lemma}
\newtheorem{proposition}{Proposition}
\theoremstyle{definition}

\newtheorem{remark}{Remark}

\newcommand{\doi}[1]{{\scriptsize 
\textsc{doi}: \href{http://dx.doi.org/#1}{\nolinkurl{#1}}}}

\title[Lyapunov functions for fractional-order systems in biology]{Lyapunov 
functions for fractional-order systems in biology: methods and applications}

\author[Boukhouima, Hattaf, Lotfi, Mahrouf, Torres and Yousfi]{}



\email{adnaneboukhouima@gmail.com}
\email{k.hattaf@yahoo.fr}
\email{lotfiimehdi@gmail.com}
\email{marouane.mahrouf@gmail.com}
\email{delfim@ua.pt}
\email{nourayousfi@hotmail.com}

\thanks{This work is part of first author's Ph.D., 
carried out at the University of Hassan II Casablanca.}

\thanks{$^*$Corresponding author: Delfim F. M. Torres (delfim@ua.pt)}

\AtBeginDocument{{\noindent\small
This is a preprint of a paper whose final and definite form is with 
\emph{Chaos, Solitons and Fractals}, available at \url{https://doi.org/10.1016/j.chaos.2020.110224}.\\
Submitted 25-May-2020; Revised 25-July-2020; Accepted for publication 19-August-2020.}
\vspace{9mm}}

\begin{document}

\maketitle

\centerline{\scshape Adnane Boukhouima$^a$, Khalid Hattaf$^{a,b}$, El Mehdi Lotfi$^a$,}
\centerline{\scshape Marouane Mahrouf$^a$, Delfim F. M. Torres$^{c,*}$ and Noura Yousfi$^a$}
\medskip
{\footnotesize
\centerline{$^a$Laboratory of Analysis, Modeling and Simulation (LAMS)}
\centerline{Faculty of Sciences Ben M'sik, Hassan II University}
\centerline{P.O. Box 7955 Sidi Othman, Casablanca, Morocco}
\smallskip
\centerline{$^b$Centre R\'egional des M\'etiers de l'Education et de la Formation (CRMEF)}
\centerline{20340 Derb Ghalef, Casablanca, Morocco}
\smallskip
\centerline{$^c$Center for Research and Development in Mathematics and Applications (CIDMA)}
\centerline{Department of Mathematics, University of Aveiro, 3810-193 Aveiro, Portugal}
}


\begin{abstract}
We prove new estimates of the Caputo derivative of order 
$\alpha \in (0,1]$ for some specific functions. The estimations 
are shown useful to construct Lyapunov functions for systems 
of fractional differential equations in biology, based on those 
known for ordinary differential equations, and therefore useful
to determine the global stability of the equilibrium points 
for fractional systems. To illustrate the usefulness of our 
theoretical results, a fractional HIV population model
and a fractional cellular model are studied.
More precisely, we construct suitable Lyapunov functionals 
to demonstrate the global stability of the free and endemic 
equilibriums, for both fractional models, and we also perform 
some numerical simulations that confirm our choices. 
\end{abstract}


\section{Introduction}

Fractional calculus (FC) is the mathematical theory that generalizes 
the integrals and derivatives to real or complex order \cite{MR3561379}. 
During the last decades, FC has gained popularity and importance in diverse fields of science 
and engineering, including biology, physics, chemistry, engineering, finance, 
and control theory, see, e.g., \cite{Capponetto,Cole,Debnath,MR3856654,Hilfer,Scalas,Yuste}. 
Recently, several works have appeared in the literature that deal with various 
applications of FC in real-life problems. In \cite{Jajarmi}, the authors discuss 
general fractional optimal control problems (FOCPs) involving fractional derivatives (FD) 
with singular and non-singular kernels. They derive necessary optimality conditions 
and propose an efficient method for solving, numerically, these problems. 
Y{\i}ld{\i}z et al. \cite{yildiz2020new} formulate new time FOCPs governed 
by Caputo--Fabrizio FD. To solve these problems, they firstly convert them  
into Volterra-type systems and then apply a new numerical scheme based 
on the approximation of Volterra integrals. Since most fractional-order problems 
cannot be solved explicitly, their numerical simulations are of crucial importance
to researchers. In the work of Veeresha et al. \cite{veeresha2020efficient}, 
the existence of solutions for fractional generalized Hirota--Satsuma 
coupled Korteweg-de-Vries (KdV) and coupled modified KdV equations 
are investigated with the aid of a fractional natural decomposition method. 
In a similar way, Singh et al. \cite{singh2020efficient} perform 
a comparison between the reduced differential transform method 
and the local fractional series expansion method for solving 
local fractional Fokker--Planck equations on the Cantor set. 
They confirm that the two proposed methods are very successful 
and simple to solve differential equations with fractional derivative 
operators of local nature. The authors of \cite{Kumar} extend 
the fractional vibration equation for very large membranes, 
with distinct special cases, by considering the Atangana--Baleanu 
fractional derivative. They also employ a numerical algorithm, based 
on the homotopy technique, to examine the fractional vibration equation. 
For that, they show the effects of space, time, and order of the Atangana‐-Baleanu 
derivative on the graphical displacement, and confirm that the Atangana‐-Baleanu 
fractional derivative is very efficient in describing vibrations in large membranes.   

The advantage of fractional differentiation is that it provides 
a powerful tool to model real-world processes with long-range memory, 
long-range interactions, and hereditary properties, which exist 
in most biological systems, as opposed to integer-order differentiation, 
where such effects are neglected \cite{MR3787674,Rossikhin,Saeedian,Zaslavsky}. 
For these reasons, modeling with fractional differential equations (FDEs) 
has attracted the interest of researchers in biology 
\cite{MR3999697,MR4103308,MR4074117}. 
In \cite{jajarmi2019new}, the authors investigate 
a fractional version of the SIRS model for the human 
respiratory syncytial virus (HRSV) disease, involving a new derivative 
operator with Mittag--Leffler kernel in the Caputo sense. They confirm, 
from simulations, that fractional modeling is less costly and more effective, 
than the proposed approach in the classical version of the model, 
to diminish the number of HRSV infected individuals. 
Sajjadi et al. \cite{sajjadi2020new} analyze hyperchaotic behaviors 
of a biological snap oscillator and study the chaos control 
and synchronization via a fractional-order model.  
They conclude that fractional calculus leads to more realistic 
and flexible models with memory effects, which could help to design 
more efficient controllers. Singh et al. \cite{singh2020new} analyze 
the dynamical behavior of a fish farm model, related with the 
Atangana--Baleanu (AB) derivative. They discuss the influence of the derivative 
order on nutrients, fish, and mussels, and show that when this order tends to one, 
then the AB derivative gives interesting results. Elettreby et al. \cite{Elettreby} 
propose a fractional-order species model to study the interaction of a system 
that consists of two-prey and one-predator. They study the stability of the 
equilibria and prove that the coexistence equilibrium points are stable 
without any conditions, in contrast with the corresponding
ordinary differential equations (ODEs) model, where some conditions 
are imposed for the stability of the same points. This means that FDEs 
have a larger stability region than those of ODEs \cite{Elettreby}. 

FDEs have been also successfully applied in epidemiology, 
as well as in virology \cite{Area,MR3872489,MyID:426}. 
Huo et al. \cite{Huo} proposed a fractional 
homogeneous-mixing population model for human immunodeficiency virus (HIV), 
which incorporates anti-HIV preventive vaccines, and studied the backward 
bifurcation of the equilibrium points. They also generalize the integer-order 
LaSalle invariant theorem for fractional-order systems and demonstrate
the global stability of the disease-free equilibrium point \cite{Huo}. 

In the work \cite{Wojtak} of Wojtak et al., the authors investigate
the uniform asymptotic stability of the unique endemic equilibrium 
for a Caputo fractional-order tuberculosis (TB) model. They confirm
that the proposed fractional-order model provides richer and more 
flexible results when compared with the corresponding integer-order 
TB model \cite{MyID:426,Wojtak}. 

In \cite{Parra}, Gonz\'{a}lez-Parra et al. propose a nonlinear 
fractional-order model to explain, and help to understand, the outbreak 
of influenza A(H1N1) worldwide. They show that the fractional-order 
model gives wider peaks and leads to better approximations 
of the real epidemic data \cite{Parra}. 

Rihan et al. \cite{Rihan} develop a fractional-order 
model for hepatitis C dynamics, in order to describe 
the interactions between healthy liver cells $H$, 
infected liver $I$, and virus load $V$. They confirm that the proposed 
model gives consistent results with the reality 
of the interactions \cite{Rihan}. 

In the study of Arafa et al. \cite{Arafa}, the authors compare
the results of the fractional-order model with the ones from 
the integer/classical model, taking into account real data obtained 
from 10 patients during primary HIV infection. They prove that 
the results of the fractional-order model give better predictions 
to the plasma virus load of the patients than those 
of the integer-order model \cite{Arafa}. 

In \cite{singh2018analysis}, the authors study diabetes 
and its complications with the help of the Caputo--Fabrizio 
fractional derivative. They observe, via numerical simulations, 
that when the derivative order is near to one, then the 
Caputo--Fabrizio non-integer order derivative reveals better
absorbing characteristics. For other related works, see, e.g., 
\cite{Area,Boukhouima,Cardoso,Kumar2020new,Li1,Rihan1,singh2017new}.

Stability analysis of FDEs through Lyapunov functions is 
investigated by Delavari et al. \cite{Li}. 
Their method requires to construct a suitable function, 
which is not easy to find in the fractional case. 
In mathematical biology, the stability of equilibrium points, 
via Lyapunov method, is a very effective way to determine 
the global behavior of a system without solving it analytically. 
This is based on the construction of well-chosen Lyapunov functions, 
according to the nature of the system under study. In the literature, 
the most well-known Lyapunov functions are quadratic 
and Volterra-type functions. Accordingly, Aguila-Camacho et al. 
\cite{Aguila-Camacho} extend such quadratic functions to the fractional case 
and then study the stability of fractional-order time-varying systems. 
In 2015, Duarte-Mermoud et al. \cite{Duarte-Mermoud} generalized 
the result of \cite{Aguila-Camacho} to the vector case, in order 
to prove the stability of fractional-order models with reference 
adaptive control schemes. 

Vargas-De-Le\'{o}n uses Volterra-type Lyapunov functions to determine 
the stability of several fractional-order epidemic systems \cite{Vargas}. 
However, quadratic and Volterra-type Lyapunov functions 
can be successfully used to demonstrate the stability 
of the equilibrium points only in particular cases,
and further work is needed.

Motivated by these works, we prove here a new result that estimates 
the Caputo fractional derivative for certain specific functions. 
Based on our result, we are able to construct Lyapunov functions
for systems of FDEs in biology by using the Lyapunov functions 
of corresponding systems formulated by ODEs and, subsequently, 
to establish the global asymptotic stability of constant steady-state solutions.

The paper is organized as follows. In Section~\ref{sec:02}, new inequalities 
to estimate the FD of order $ \alpha\in(0,1]$,  for specific functions, 
are rigorously proved and a detailed description of the proposed method 
is presented with proofs. Then, in Section~\ref{sec:03}, 
we apply our method to study the asymptotic stability of two models 
in virology and epidemiology. We end up with Section~\ref{sec:04} 
of conclusions. 


\section{Description of the method}
\label{sec:02}

Consider an $n$-dimensional autonomous system 
formulated by ordinary differential equations,
\begin{align}
\label{1}
\begin{cases}
\dfrac{du}{dt}=f(u),\\
u(t_{0})=u_{0},
\end{cases}
\end{align}
where $u$ is a non-negative vector of concentration $u_{1},\ldots,u_{n}$ 
and $f:\mathbb{R}^{n}\to \mathbb{R}^{n}$ is a $C^{1}$-function.

Let $V(u)$ be a $C^{1}$-function defined on some domain in $\mathbb{R}^{n}_{+}$. 
When $u(t)$ is a solution of (\ref{1}), it is often necessary to 
compute the time derivative of $ V(u(t))$:
\begin{equation*}
\begin{split}
\dfrac{dV\left(u(t)\right)}{dt}
&=\nabla V\left(u(t)\right)\cdot \dfrac{du(t)}{dt}\\
&=\nabla V\left(u(t)\right)\cdot f(u(t)).
\end{split}
\end{equation*}
We assume that the range of $u(t)$ is contained in the domain of $V(\cdot)$. 
The right-hand side is given by the gradient of function $V(\cdot)$ and the 
vector field $f(\cdot)$. Thus, the right-hand side is defined without the fact 
that $u(t)$ is a solution of (\ref{1}), which is important 
for our construction of Lyapunov functions.

In the literature, many authors define explicit Lyapunov 
functions of the form
\begin{equation}
\label{3}
V(u)=\sum^{n}_{i=1}a_{i}\Psi_{i}(u_{i})
\end{equation}
with $a_{i}>0$ and
\begin{equation}
\label{4}
\begin{split}
\Psi_{i}(u_{i})&=\int^{u_{i}}_{u_{i}^{*}}
\dfrac{g_{i}(s)-g_{i}(u_{i}^{*})}{g_{i}(s)}ds\\
&=u_{i}-u_{i}^{*}-\int^{u_{i}}_{u_{i}^{*}}
\dfrac{g_{i}(u_{i}^{*})}{g_{i}(s)}ds,
\quad \forall i=1,\ldots n,
\end{split}
\end{equation}
where $u^{*}=( u^{*}_{1},\cdots,u^{*}_{n})$ is any equilibrium 
of (\ref{1}), $u^{*}_{i}> 0$ for all $1\leq i \leq n$, 
and $g_{i}$ is a non-negative, differentiable, 
and strictly increasing function on $\mathbb{R}^{+}$, 
see, e.g., \cite{Elaiw,Elaiw1,Georgescu,Hattaf1}.

\begin{remark}
If $ u^{*}_{i}=0 $, then function $ \Psi_{i}(u_{i})$ reduces to
\begin{equation*}
\Psi_{i}(u_{i})= u_{i}.
\end{equation*}
\end{remark}

\begin{remark}
If $ g_{i}(s)=s$, then function $ \Psi_{i}(u_{i}) $ becomes
\begin{equation*}
\Psi_{i}(u_{i})= u_{i}-u_{i}^{*}
-u_{i}^{*}\ln \left( \dfrac{u_{i}}{u_{i}^{*}}\right).
\end{equation*}
\end{remark}

It is easy to see that function $\Psi_{i}$ is strictly positive 
in $\mathbb{R}^{+}\setminus \lbrace u_{i}^{*}\rbrace$ with 
$\Psi_{i}(u_{i}^{*})=0$. In fact, $\Psi_{i}$ is differentiable and 
$$ 
\dfrac{d\Psi_{i} }{du_{i}}=1-\dfrac{g_{i}(u_{i}^{*})}{g_{i}(u_{i})}. 
$$
Since $g_{i}$ is a strictly increasing function, then $\Psi_{i}$ 
is strictly decreasing if $u_{i}< u_{i}^{*}$ and strictly increasing 
if $u_{i}> u_{i}^{*} $, with $ u_{i}^{*} $ its global minimum. 
In this case, we have
\begin{equation}
\label{5}
\dfrac{dV(u(t))}{dt}=\sum^{n}_{i=1}a_{i}\left( 
1-\dfrac{g_{i}(u_{i}^{*})}{g_{i}(u_{i})}\right)f_{i}(u_{i}),
\end{equation}
where 
$$ 
f(u)=\left(f_{1}(u_{1}),\ldots,f_{n}(u_{n})\right)^{T}.
$$
On the other hand, let us consider the following general 
type of fractional-order system:
\begin{align}
\label{6}
\begin{cases}
_{t_{0}}^{C}D_{t}^{\alpha}u(t)=f\left(u(t)\right),
\qquad \alpha \in (0,1],\\
u(t_{0})=u_{0},
\end{cases}
\end{align}
where $_{t_{0}}^{C}D_{t}^{\alpha}$ denotes the Caputo fractional derivative 
of order $\alpha$, defined for the function $u$ by 
\begin{equation}
\label{df}
_{t_{0}}^{C}D_{t}^{\alpha}u(t)
= \dfrac{1}{\Gamma(1-\alpha)}\int^{t}_{t_0}\dfrac{u'(y)}{(t-y)^{\alpha}}dy
\end{equation}
(see, e.g., \cite{Podlubny}). Here, $'$ denotes $\dfrac{d}{dy}$.
We note that system (\ref{6}) has the same equilibrium points as system (\ref{1}).

The Caputo fractional derivative of $V$ along the solution of (\ref{6}) is given by
\begin{equation}
\label{7}
_{t_{0}}^{C}D_{t}^{\alpha}V(u(t))
=\sum^{n}_{i=1}a_{i} \: _{t_{0}}^{C}D_{t}^{\alpha}\Psi_{i}(u_{i}).
\end{equation}
To extend the Lyapunov functions (\ref{3}) to the Caputo fractional-order 
system (\ref{6}), through an inequality that estimates the Caputo fractional 
derivative of these functions, we prove the following lemma, 
which is the main result of this section.

\begin{lemma}
\label{Lemma 3.2}
Let $x(t)\in\mathbb{R}^{+}$ be a continuous and differentiable function. 
Then, for any $t \geq t_{0}$, $0<\alpha\leq1$, and $\bar{x}>0$, we have
\begin{equation}
\label{8}
_{t_{0}}^{C}D_{t}^{\alpha}\Psi(x(t))
\leq \left( 1-\dfrac{g(\bar{x})}{g(x(t))}\right)
\: _{t_{0}}^{C}D_{t}^{\alpha}x(t),
\end{equation}
where
\begin{equation*}
\Psi(x)=x-\bar{x}-\int^{x}_{\bar{x}}\dfrac{g(\bar{x})}{g(s)}ds,
\end{equation*}
with $g:\mathbb{R}^{+}\to \mathbb{R}^{+}$ a differentiable 
and strictly increasing function.
\end{lemma}

\begin{proof}
We start by reformulating inequality (\ref{8}). 
By the linearity of the Caputo fractional derivative, 
we obtain that
\begin{equation*}
_{t_{0}}^{C}D_{t}^{\alpha}\Psi(x(t))
=\:_{t_{0}}^{C}D_{t}^{\alpha}x(t)-\:_{t_{0}}^{C}D_{t}^{\alpha}\left[ 
\int_{\bar{x}}^{x(t)}\dfrac{g(\bar{x})}{g(s)}ds\right].
\end{equation*}
Hence, inequality (\ref{8}) becomes
\begin{equation*}
_{t_{0}}^{C}D_{t}^{\alpha}x(t)-\:_{t_{0}}^{C}D_{t}^{\alpha}\left[ 
\int_{\bar{x}}^{x(t)}\dfrac{g(\bar{x})}{g(s)}ds\right]
\leq \left( 1-\dfrac{g(\bar{x})}{g(x)}\right)
\: _{t_{0}}^{C}D_{t}^{\alpha}x(t).
\end{equation*}
Because $g$ is a non-negative function, we get
\begin{equation*}
g(x(t))_{t_{0}}^{C}D_{t}^{\alpha}x(t)
-g(x(t))_{t_{0}}^{C}D_{t}^{\alpha}\left[ 
\int_{\bar{x}}^{x(t)}\dfrac{g(\bar{x})}{g(s)}ds\right]
\leq g(x(t))_{t_{0}}^{C}D_{t}^{\alpha}x(t)
-g(\bar{x})_{t_{0}}^{C}D_{t}^{\alpha}x(t).
\end{equation*}
Thus,  
\begin{equation}
\label{10}
_{t_{0}}^{C}D_{t}^{\alpha}x(t)-g(x(t))_{t_{0}}^{C}D_{t}^{\alpha}\left[ 
\int_{\bar{x}}^{x(t)}\dfrac{1}{g(s)}ds\right]
\leq 0.
\end{equation}
Using the definition of Caputo fractional derivative (\ref{df}), we have
\begin{equation*}
_{t_{0}}^{C}D_{t}^{\alpha}x(t)
=\dfrac{1}{\Gamma(1-\alpha)}
\int^{t}_{t_{0}}\dfrac{x'(y)}{(t-y)^{\alpha}} dy
\end{equation*}
and
\begin{equation*}
_{t_{0}}^{C}D_{t}^{\alpha}\left[ 
\int_{\bar{x}}^{x(t)}\dfrac{1}{g(s)}ds\right]
=\dfrac{1}{\Gamma(1-\alpha)}\int^{t}_{t_{0}}
\dfrac{x'(y)}{(t-y)^{\alpha}g(x(y))} dy.
\end{equation*}
Consequently, the inequality (\ref{10}) can be written as follows:
\begin{equation}
\label{11}
\dfrac{1}{\Gamma(1-\alpha)}
\int^{t}_{t_{0}}\dfrac{x'(y)}{(t-y)^{\alpha}}\left( 
1-\dfrac{g(x(t))}{g(x(y))}\right) dy
\leq 0.
\end{equation}
Now, we show that the inequality (\ref{11}) 
is verified. Denoting
$$
H(t)=\dfrac{1}{\Gamma(1-\alpha)}\int^{t}_{t_{0}}
\dfrac{x'(y)}{(t-y)^{\alpha}}\left( 1-\dfrac{g(x(t))}{g(x(y))}\right) dy,
$$ 
we integrate by parts by defining
\begin{equation*}
v(y)=\dfrac{(t-y)^{-\alpha}}{\Gamma(1-\alpha)}; 
\end{equation*}
$$
v'(y)=\dfrac{\alpha(t-y)^{-(\alpha+1)}}{\Gamma(1-\alpha)};
$$
and
\begin{equation*}
\:w(y)=x(y)-x(t)-\int^{x(y)}_{x(t)}\dfrac{g(x(t))}{g(s)}ds; 
\quad w'(y)=x'(y)\left( 1-\dfrac{g(x(t))}{g(x(y))}\right);
\end{equation*}
to obtain
\begin{equation}
\label{12}
\begin{split}
H(t)
&= \left[ \dfrac{(t-y)^{-\alpha}}{\Gamma(1-\alpha)}\left(x(y)-x(t) 
-\int^{x(y)}_{x(t)}\dfrac{g(x(t))}{g(s)}ds\right) \right]^{y=t}\\
&\quad -\dfrac{(t-t_{0})^{-\alpha}}{\Gamma(1-\alpha)}\left(x(t_{0})-x(t)
-\int^{x(t_{0})}_{x(t)}\dfrac{g(x(t))}{g(s)}ds\right)\\
&\quad -\int^{t}_{t_{0}}\dfrac{\alpha(t-y)^{-(\alpha+1)}}{\Gamma(1-\alpha)}\left(
x(y)-x(t)-\int^{x(y)}_{x(t)}\dfrac{g(x(t))}{g(s)}ds\right) dy.
\end{split}
\end{equation}
We can easily see that the first term in (\ref{12}) is undefined at $u=t$ 
$(\frac{0}{0})$. Let us analyze the corresponding limit. 
By L'H\^{o}pital's rule, we get
\begin{equation*}
\lim_{y\to t}\dfrac{(t-y)^{-\alpha}}{\Gamma(1-\alpha)}\left(
x(y)-x(t)-\int^{x(y)}_{x(t)}\dfrac{g(x(t))}{g(s)}ds\right)
=\lim_{y\to t}\dfrac{x'(y)\left( 1
-\dfrac{g(x(t))}{g(x(y))}\right)}{-\alpha\Gamma(1-\alpha)(t-y)^{\alpha-1}}=0.
\end{equation*}
Hence,
\begin{equation*}
\begin{split}
H(t)
&=-\dfrac{(t-t_{0})^{-\alpha}}{\Gamma(1-\alpha)}\left( x(t_{0})-x(t)
-\int^{x(t_{0})}_{x(t)}\dfrac{g(x(t))}{g(s)}ds\right)\\
&\qquad -\int^{t}_{t_{0}}\dfrac{\alpha(t-y)^{-(\alpha+1)}}{\Gamma(1
-\alpha)}\left(x(y)-x(t)-\int^{x(y)}_{x(t)}\dfrac{g(x(t))}{g(s)}ds\right) dy.
\end{split}
\end{equation*}
From (\ref{4}), we get
\begin{equation*}
H(t)=\dfrac{1}{\Gamma(1-\alpha)}
\int^{t}_{t_{0}}\dfrac{x'(u)}{(t-u)^{\alpha}}\left( 
1-\dfrac{g(x(t))}{g(x(u))}\right) du\leq 0
\end{equation*}
and, as a result, the inequality (\ref{11}) is satisfied. 
This completes the proof.
\end{proof}

A particular case of Lemma~\ref{Lemma 3.2} 
is given in the following corollary.

\begin{corollary}
Let $x(t)\in\mathbb{R}^{+}$ be a continuous and differentiable function. 
Then, for any  $t \geq t_{0}$, $0<\alpha\leq 1$, and $ \bar{x}\geq 0$, 
one has
\begin{equation*}
_{t_{0}}^{C}D_{t}^{\alpha}\left[ x(t)-\bar{x}
-\bar{x}\ln\dfrac{x(t)}{\bar{x}}\right] 
\leq \left( 1-\dfrac{\bar{x}}{x(t)}\right)
\: _{t_{0}}^{C}D_{t}^{\alpha}x(t).
\end{equation*}
\end{corollary}

\begin{proof}
Define function $g$ on $[0,+\infty)$ by $g(s)=s$. 
Obviously, function $g$ is a non-negative and 
strictly increasing function on $[0,+\infty)$ with
\begin{equation*}
\psi(x(t))
=x(t)-\bar{x}-\bar{x}\ln\dfrac{x(t)}{\bar{x}}.
\end{equation*}
The result follows by Lemma~\ref{Lemma 3.2}.
\end{proof}

\begin{remark}
We can see that the inequality obtained for Volterra-type 
Lyapunov functions in \cite[Lemma~3.1]{Vargas} 
is a special case of our Lemma~\ref{Lemma 3.2}.
\end{remark}

Finally, using Lemma~\ref{Lemma 3.2}, we estimate the Caputo 
fractional derivative of $V$ in (\ref{7}) through the following inequality:
\begin{equation*}
\begin{split}
_{t_{0}}^{C}D_{t}^{\alpha}V(u(t))
&\leq \sum^{n}_{i=1}a_{i}\: \left( 1
-\dfrac{g_{i}(u_{i}^{*})}{g_{i}(u_{i})}\right)
\: _{t_{0}}^{C}D_{t}^{\alpha}u_{i}(t)\\
&=\sum^{n}_{i=1}a_{i}\: \left( 1
-\dfrac{g_{i}(u_{i}^{*})}{g_{i}(u_{i})}\right)f_{i}(u_{i}).
\end{split}
\end{equation*}
We summarize the above discussion in the following proposition.

\begin{proposition}
\label{prop:2.3}
If $V$ is a Lyapunov function for the ordinary differential equation
(\ref{1}) of the form described by (\ref{3}), then $V$ is also a 
Lyapunov function for the Caputo fractional-order system (\ref{6}).
\end{proposition}

In other cases, some authors constructed a Lyapunov function
for system (\ref{1}) given by the composition of $V$ and a 
quadratic function $Q$, that is, of the form
\begin{align}
\label{quad functions}
L(u(t))=V(u(t))+Q(u(t)),
\end{align}
where 
\begin{equation*}
Q(u)=\sum_{i=1}^{n}\dfrac{b_{i}}{2}(u_{i}-u_{i}^{*})^{2}, 
\end{equation*}
with $b_{i} \geq 0$ \cite{Cervantes,Buonomo,Maziane}. 
The time derivative of $L$ is given by
\begin{align}
\label{quad functions 1}
\dfrac{dL(u(t))}{dt}=\sum^{n}_{i=1}\left[  a_{i}\left( 
1-\dfrac{g_{i}(u_{i}^{*})}{g_{i}(u_{i})}\right)
+b_{i}(u_{i}-u_{i}^{*})\right] f_{i}(u_{i}).
\end{align}
Therefore, computing the fractional time derivative of $L$ 
by using our Lemma~\ref{Lemma 3.2} and Lemma~1 of 
\cite{Aguila-Camacho}, we obtain that
\begin{equation*}
\begin{split}
_{t_{0}}^{C}D_{t}^{\alpha}L(u(t))
&=\:_{t_{0}}^{C}D_{t}^{\alpha}V(u(t)) +_{t_{0}}^{C}D_{t}^{\alpha}Q(u(t))\\
&\leq\sum^{n}_{i=1}\left[  a_{i}\left( 
1-\dfrac{g_{i}(u_{i}^{*})}{g_{i}(u_{i})}\right)
+b_{i}(u_{i}-u_{i}^{*})\right] f_{i}(u_{i}).
\end{split}
\end{equation*}
Thus, the following result holds.

\begin{corollary}
\label{cor24}
If $L$ is a Lyapunov function for the ordinary 
differential equation (\ref{1}) of the form described by 
(\ref{quad functions}), then $L$ is also a Lyapunov function
for the Caputo fractional-order system (\ref{6}).
\end{corollary}

Let $D$ be a bounded closed set in $\mathbb{R}^n$. 
Assume that the largest invariant set in 
$$
\left\{u\in D\ | \ \ _{t_{0}}^{C}D_{t}^{\alpha}L(u(t))=0\right\}
$$ 
is just the singleton $\{u^*\}$. Then, 
we get the following result.

\begin{proposition}
\label{prop:2.5}
If (\ref{5}) (respectively (\ref{quad functions 1})) is non-positive, 
then $_{t_{0}}^{C}D_{t}^{\alpha}V(u(t))\leq 0$ (respectively, 
$_{t_{0}}^{C}D_{t}^{\alpha}L(u(t))\leq 0 $). It follows that 
the positive equilibrium $u^*$ of the fractional-order system
(\ref{6}) is globally asymptotically stable.
\end{proposition}

\begin{proof}
By Lemma~4.6 of \cite{Huo}, every solution originating in $D$ 
tends to the largest invariant set of 
$$
\left\{u \in D |\ \ _{t_{0}}^{C}D_{t}^{\alpha}L(u(t))=0\}=\{u \in D |\ \ u=u^*\right\}.
$$
Thus, 
$$
\displaystyle{\lim_{t\to+\infty}u(t)=u^*}.
$$ 
This completes the proof.
\end{proof}


\section{Applications}
\label{sec:03}

In this section, we apply our method to study the stability 
of two fractional-order biological models. Our procedure 
is based on the construction of Lyapunov functions for 
FDEs using Lyapunov functions for ODEs.


\subsection{Example 1: an HIV population model}

In this example, we consider the SICA model of Silva and Torres \cite{Silva}, 
which contains four variables: the susceptible individuals ($S$), 
HIV-infected individuals with no clinical symptoms of AIDS ($I$), 
HIV-infected individuals under ART treatment ($C$), 
and HIV-infected individuals with AIDS clinical symptoms ($A$). 
The model is given by the following nonlinear 
system of differential equations:
\begin{equation}
\label{sys22}
\begin{cases}
\dfrac{dS}{dt}=\Lambda-\mu S(t)-\beta S(t)I(t),\\[5pt]
\dfrac{dI}{dt}= \beta S(t)I(t)-(\rho+\phi+\mu)I(t)+\alpha A(t)+\omega C(t),\\[5pt]
\dfrac{dC}{dt}=\phi I(t)-(\omega+\mu) C(t),\\[5pt]
\dfrac{dA}{dt}=\rho I(t)-(\alpha+\mu+d)A(t).
\end{cases}
\end{equation} 
The basic reproduction number of system (\ref{sys22}), which represents the
expected average number of new HIV infections produced by a
single HIV-infected individual when in contact with a completely
susceptible population, is given by
$$ 
R_{0}=\dfrac{\beta\xi_{1}\xi_{2}}{\mathcal{N}},
$$
where $\xi_{1}=\alpha+\mu+d$, $\xi_{2}=\omega+\mu$, 
and $\mathcal{N}=\mu[\xi_{2}(\rho+\xi_{1})+\xi_{1}\phi+\rho d]+\rho \omega d$.
Silva and Torres proved that if $R_{0}>1$, then system (\ref{sys22}) has 
an endemic equilibrium $E^{*}=\left(S^{*},I^{*},C^{*},A^{*}\right)$, which is globally 
asymptotically stable \cite{Silva}. Their proof is given by constructing 
a Lyapunov function for system (\ref{sys22}) at $E^{*}$ as follows:
\begin{align*}
V_{1}(S,I,C,A)&= \Psi_{1}(S)+\Psi_{2}(I)
+\dfrac{\omega}{\xi_{2}}\Psi_{3}(C)+\dfrac{\alpha}{\xi_{1}}\Psi_{4}(A)\\
&=S-S^{*}-\int^{S}_{S^{*}}\dfrac{S^{*}}{X}dX + I
-I^{*}-\int^{I}_{I^{*}}\dfrac{I^{*}}{X}dX\\
&\qquad +\dfrac{\omega}{\xi_{2}}\left( C-C^{*}
-\int^{C}_{C^{*}}\dfrac{C^{*}}{X}dX\right) 
+\dfrac{\alpha}{\xi_{1}}\left( A-A^{*}
-\int^{A}_{A^{*}}\dfrac{A^{*}}{X}dX\right).
\end{align*}
The time derivative of $V_{1}$ is computed as
\begin{align*}
\dfrac{dV_{1}}{dt}
&=(\beta I^{*} S^{*}+\mu S^{*})\left( 2-\dfrac{S}{S^{*}}
-\dfrac{S^{*}}{S}\right)+\alpha A^{*}\left( 
2-\dfrac{AI^{*}}{A^{*}I}-\dfrac{A^{*}I}{AI^{*}}\right)\\
&\qquad +\omega C^{*}\left( 2-\dfrac{CI^{*}}{C^{*}I}
-\dfrac{C^{*}I}{CI^{*}}\right) \leq 0.
\end{align*}
However, when $R_{0}<1 $, the global stability of the disease-free 
equilibrium 
$$
E_{f} = (S_{0},0,0,0), 
$$
where $S_{0}=\frac{\Lambda}{\mu}$, 
was determined without using a Lyapunov function. Here we discuss 
the global stability of $E_{f}$ when $R_0\leq 1$. For this, we 
construct a Lyapunov function for system (\ref{sys22}) at $E_{f}$:
\begin{align*}
V_{0}(S,I,C,A)
&= \Psi_{1}(S)+\Psi_{2}(I)+\dfrac{\omega}{\xi_{2}}\Psi_{3}(C)
+\dfrac{\alpha}{\xi_{1}}\Psi_{4}(A)\\
&=S-S_{0}-\int^{S}_{S_{0}}\dfrac{S_{0}}{X}dX + I
+\dfrac{\omega}{\xi_{2}}C +\dfrac{\alpha}{\xi_{1}}A.
\end{align*}
The time derivative of $ V_{0} $ along the solutions 
of system (\ref{sys22}) satisfies
\begin{align*}
\dfrac{dV_{0}}{dt}
&=\left(1-\dfrac{S_0}{S}\right)\dfrac{dS}{dt}+\dfrac{dI}{dt}
+\dfrac{\omega}{\xi_2}\dfrac{dC}{dt}+\dfrac{\alpha}{\xi_1}\dfrac{dA}{dt}\\
&\leq -\dfrac{\mu (S-S_{0})^{2}}{S}
+\dfrac{\mathcal{N}}{\xi_{1}\xi_{2}}I(R_{0}-1).
\end{align*}
Therefore, $\dfrac{dV_{0}}{dt}\leq0$ if $ R_{0}\leq 1$. Furthermore, 
the largest compact invariant set in 
$$
\left\{(S,I,C,A)\mid\frac{d V_0}{dt}=0\right\}
$$ 
is just the singleton $E_f$. Using LaSalle's invariance principle \cite{LaSalle}, 
we conclude that $E_f$ is globally asymptotically stable.

Now, we propose the following fractional-order SICA model defined by
\begin{equation}
\label{sys24}
\begin{cases}
_{0}^{C}D_{t}^{\theta}S(t)
=\Lambda-\mu S(t)-\beta S(t)I(t),\\
_{0}^{C}D_{t}^{\theta}I(t)
= \beta S(t)I(t)-(\rho+\phi+\mu)I(t)+\alpha A(t)+\omega C(t),\\
_{0}^{C}D_{t}^{\theta}C(t)
=\phi I(t)-(\omega+\mu) C(t),\\
_{0}^{C}D_{t}^{\theta}A(t)
=\rho I(t)-(\alpha+\mu+d)A(t),
\end{cases}
\end{equation}
where $0<\theta\leq 1$, subject to the initial conditions
\begin{equation}
\label{inini}
S(0)\geq 0,\quad I(0)\geq 0,\quad C(0)\geq 0, \quad A(0)\geq 0.
\end{equation}

\begin{remark}
Following \cite{Boukhouima1}, one can easily prove that system 
\eqref{sys24}--\eqref{inini} has a unique solution for any $t>0$.
\end{remark}

Applying Proposition~\ref{prop:2.5}, we have 
$$
_{0}^{C}D_{t}^{\theta}V_0(S,I,C,A)\leq 0,
\quad \text{ when } R_0\leq 1 
$$
and 
$$
_{0}^{C}D_{t}^{\theta}V_1(S,I,C,A)\leq 0,
\quad \text{ when } R_0 > 1.
$$ 
Then, the following result holds.

\begin{theorem}
\label{thm}
Suppose that $0<\theta\leq1$.
\begin{enumerate}
\item[(i)] If $R_0\leq1$, then the disease-free equilibrium 
$E_f$ is globally asymptotically stable.
\item[(ii)] If $R_0>1$, then the endemic equilibrium 
$E^*$ is globally asymptotically stable.
\end{enumerate}
\end{theorem}

\begin{proof}
(i) Obviously, the largest invariant set in 
$$
\lbrace(S,I,C,A) \in \mathbb{R}^{4}_{+}\mid 
\, _{0}^{C}D_{t}^{\theta}V_0(S,I,C,A)=0\rbrace
$$ 
is just the singleton $ E_f $. By LaSalle's invariance principle 
in \cite{Huo}, $E_f$ is globally asymptotically stable.

\noindent (ii) It is easy to see that the largest invariant set in 
$$
\lbrace(S,I,C,A) \in \mathbb{R}^{4}_{+}\mid 
\, _{0}^{C}D_{t}^{\theta}V_1(S,I,C,A)=0\rbrace
$$ 
is just the singleton $E^*$. Using LaSalle's invariance principle, 
we conclude that $E_f$ is globally asymptotically stable.
\end{proof}

Finally, we present some numerical simulations to illustrate the stability 
results of model (\ref{sys24})--(\ref{inini}), for different values of 
$\theta$. We consider the following parameter values:
\begin{gather*}
\Lambda=10724, \quad \mu=1/69.54, \quad \beta=0.066, 
\quad \rho=0.1,\\ 
\phi=1, \quad \alpha= 0.33, \quad \omega=0.09, \quad d=1.
\end{gather*}
A direct calculation gives $R_0 = 0.2900$, which satisfies 
item (i) of Theorem~\ref{thm}. Then, the disease-free equilibrium 
$E_f = (7.4575\times 10^5, 0, 0, 0)$ 
is globally asymptotically stable, 
which leads to the eradication of HIV and AIDS from the population. 
Numerical simulations illustrate this result (see Figure~\ref{fig1}).
\begin{center}
\begin{figure}[ht!]
\includegraphics[scale=0.7]{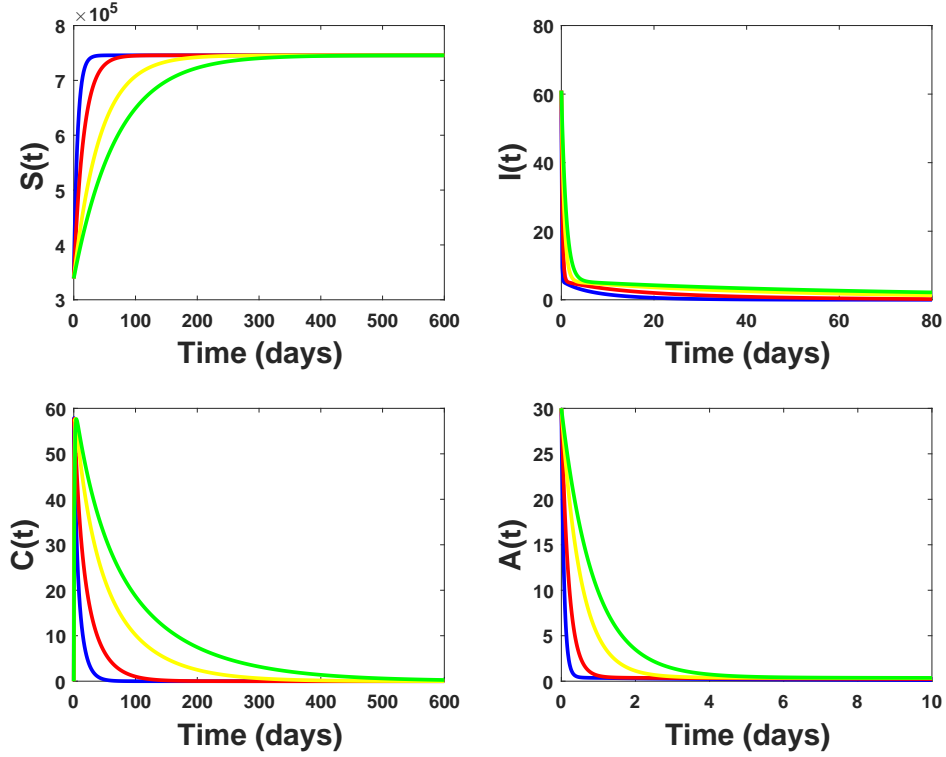}
\captionof{figure}{Stability of the disease-free equilibrium 
$E_{f}$ for the fractional-order SICA model \eqref{sys24} 
with $\theta=0.5$ (blue), $\theta=0.7$ (red), $\theta=0.9$ (yellow), 
and $\theta=1$ (green).}
\label{fig1}
\end{figure}
\end{center}
In Figure~\ref{fig2}, we choose $\beta= 0.866$, while keeping 
the other parameter values as before. In this case, we have 
$R_0 = 3.8049$ and system (\ref{sys24})--(\ref{inini}) has 
an endemic equilibrium $E^* = (0.8909 \times 10 ^{5},4.1489
\times 10 ^{4},3.9748\times 10 ^{5},3.0861)$. 
Hence, by item (ii) of Theorem~\ref{thm}, 
$E^*$ is globally asymptotically stable, which means 
that the disease persists in the population.
\begin{center}
\begin{figure}[ht!]
\includegraphics[scale=0.7]{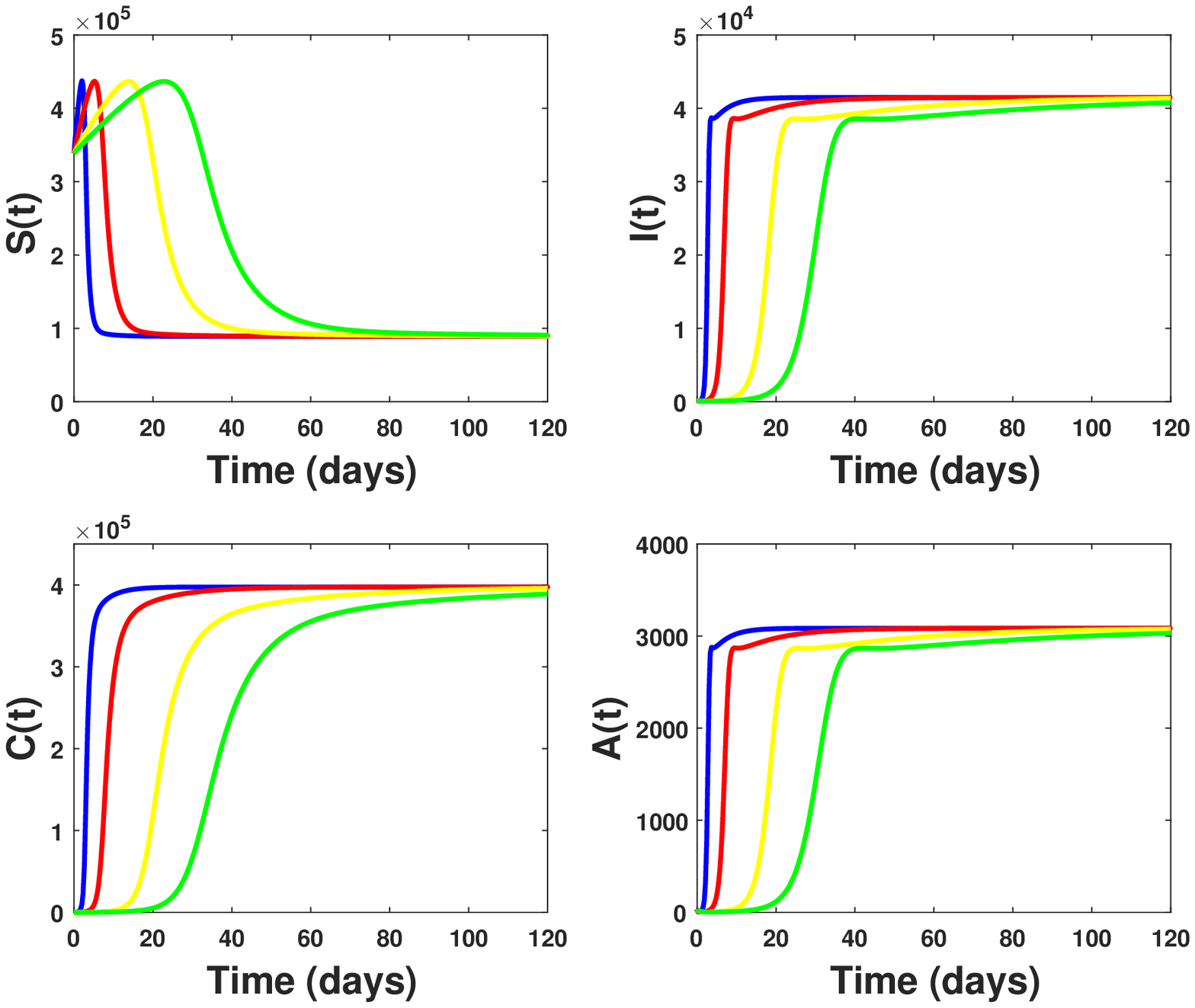}
\captionof{figure}{Stability of the endemic equilibrium $E^{*}$
for the fractional-order SICA model \eqref{sys24} with 
$\theta=0.5$ (blue), $\theta=0.7$ (red), $\theta=0.9$ (yellow), 
and $\theta=1$ (green).}
\label{fig2}
\end{figure}
\end{center}


\subsection{Example 2: an HIV cellular model}

We consider the HIV infection model with cure rate of infected cells in eclipse stage 
as proposed by Maziane et al. \cite{Maziane1}. This model contains also four variables: 
the uninfected CD4$^+$T cells $(T)$, infected cells in the eclipse stage 
(unproductive cells, denoted by $E$), productive infected cells $(I)$, 
and free virus particles $(V)$. The model is given by the following 
non-linear system of ODEs:
\begin{equation}
\label{TEIV}
\begin{cases}
\dfrac{dT(t)}{dt}=\lambda-\mu_T T(t)-f\left(T(t),V(t)\right)V(t)+\rho E(t),\\[5pt]
\dfrac{dE(t)}{dt}=f\left(T(t),V(t)\right)V(t)-(\mu_E+\rho+\gamma)E(t),\\[5pt]
\dfrac{dI(t)}{dt}=\gamma E(t)-\mu_I I(t),\\[5pt]
\dfrac{dV(t)}{dt}=kI(t)-\mu_V V(t),
\end{cases}
\end{equation}
where $\lambda$ is the recruitment rate of uninfected cells. 
The constants $\mu_T$, $\mu_E$, $\mu_I$, and $\mu_V$ represent the death rates 
of uninfected cells, unproductive cells, productive cells, and virus, respectively. 
The constant $\rho$ is the rate at which the unproductive infected cells may revert 
to the uninfected cells. The incidence of HIV infection of health CD4$^+$T cells 
has the form 
$$
f(T,V) = \frac{\beta T}{1+\alpha_1 T+\alpha_2 V+\alpha_3 TV},
$$ 
where $\beta$ is the infection rate and $\alpha_1$, $\alpha_2$, $\alpha_3\geq 0$ 
are non-negative constants. The constant $\gamma$ is the rate at which
infected cells in the eclipse stage become productive infected cells and the constant
$k$ is the rate of production of virions by infected cells.
The basic reproduction number $R_0$ is given by
\begin{equation*}
R_0=\dfrac{\lambda\beta k\gamma}{\mu_I 
\mu_V (\lambda \alpha_1+\mu_T)(\rho+\mu_E+\gamma)},
\end{equation*}
which is the average number of secondary infections produced by
one productive infected cell during the period of infection 
when all cells are uninfected. Moreover,
Maziane et al. \cite{Maziane1} show that model (\ref{TEIV}) 
is globally asymptotically stable. The proof is done by using 
the following Lyapunov function in $\mathbb{R}^4_+$:
\begin{eqnarray}
\label{L:ex2}
L(T,E,I,V)
&=& \Psi_{1}(T)+\dfrac{\rho+\mu_E+\gamma}{\gamma}\Psi_{2}(I)
+\Psi_{3}(E)+\dfrac{\mu_I(\rho+\mu_E+\gamma)}{k\gamma}\Psi_{4}(V)\nonumber\\
&\quad& +\dfrac{\rho(1+\alpha_2 \overline{V})}{2(1+\alpha_1\overline{T}
+\alpha_2\overline{V}+\alpha_3\overline{T}\overline{V})} 
\left(T-\overline{T}+E-\overline{E}\right)^2\nonumber\\
&=& T-\overline{T}-\int^{T}_{\overline{T}}
\dfrac{f(\overline{T},\overline{V})}{f(\theta,\overline{V})}d\theta
+ \dfrac{\rho+\mu_E+\gamma}{\gamma} \left(I-\overline{I}
+\int^{I}_{\overline{I}}\dfrac{\overline{I}}{\theta}d\theta\right)\\
&\quad& +E-\overline{E}+\int^{E}_{\overline{E}}\dfrac{\overline{E}}{\theta}d\theta
+\dfrac{\mu_I(\rho+\mu_E+\gamma)}{k\gamma}\left(V-\overline{V}
+\int^{V}_{\overline{V}}\dfrac{\overline{V}}{\theta}d\theta\right)\nonumber\\
&\quad& +\dfrac{\rho(1+\alpha_2 \overline{V})}{2(1+\alpha_1\overline{T}
+\alpha_2\overline{V}+\alpha_3\overline{T}\overline{V})} 
\left(T-\overline{T}+E-\overline{E}\right)^2,\nonumber
\end{eqnarray}
where $\overline{u}=(\overline{T},\overline{E},\overline{I},\overline{V})$ 
is an arbitrary equilibrium of system (\ref{TEIV}) and when 
$\overline{\ \cdot\ }$ is zero, for some equilibrium coordinate, 
the corresponding integral term vanishes.

Now, we propose the following fractional-order HIV infection model 
with cure rate of infected cells in eclipse stage:
\begin{equation}
\label{TEIVfra}
\begin{cases}
_{0}^{C}D_{t}^{\alpha}T(t)
=\lambda-\mu_T T(t)-f\left(T(t),V(t)\right)V(t)+\rho E(t),\\[5pt]
_{0}^{C}D_{t}^{\alpha}E(t)
=f\left(T(t),V(t)\right)V(t)-(\mu_E+\rho+\gamma)E(t),\\[5pt]
_{0}^{C}D_{t}^{\alpha}I(t)
=\gamma E(t)-\mu_I I(t),\\[5pt]
_{0}^{C}D_{t}^{\alpha}V(t)
=kI(t)-\mu_V V(t),
\end{cases}
\end{equation}
subject to initial conditions
\begin{equation}
\label{TEIVfraini}
T(0)\geq 0,\quad E(0)\geq 0,
\quad I(0)\geq 0,\quad V(0)\geq 0,
\end{equation}
where $0<\alpha\leq 1$. 

\begin{remark}
It follows from the results of Boukhouima et al. \cite{Boukhouima2}
that system \eqref{TEIVfra}--\eqref{TEIVfraini} has a unique global solution. 
\end{remark}

Let $u(t)=(T(t),E(t),I(t),V(t))$ be a solution of (\ref{TEIVfra})--(\ref{TEIVfraini}). 
According to our Corollary~\ref{cor24}, since $L$ given by \eqref{L:ex2} is a Lyapunov 
function for the ordinary differential equations (\ref{TEIV}) of the form described 
by (\ref{quad functions}), then $L$ is also a Lyapunov function for the 
fractional-order system (\ref{TEIVfra})--(\ref{TEIVfraini}).


\section{Conclusion}
\label{sec:04}

Mathematical models using ordinary differential equations have proved 
valuable to understand the interactions and the evolution of different 
biological phenomena \cite{MR2573923}. However, such models ignore memory 
effects and long-range interactions, which exist in most biological systems. 
For this reason, fractional differential equations have recently been used 
to model more accurately such real processes: see, e.g.,
\cite{MR4060533,MR3997113,MR3940232,MR3984125}. 
As is well known, stability analysis is an important performance metric 
for any dynamical system \cite{MR3671999,MR3783261,MR3980195}. 
The fractional-order extension of Lyapunov's 
direct method becomes, naturally, one of main interesting techniques 
to study the global behavior of fractional-order models without solving 
explicitly such systems \cite{MR2821032}. 
This method provides a way to determinate 
asymptotic stability by constructing a suitable Lyapunov function, 
which is not easy to find. Here, a new lemma for Caputo fractional 
derivatives of order $ 0<\alpha\leq 1$, of some functions, 
is presented. Our approach consists to construct Lyapunov functions for 
FDEs using Lyapunov functions for ODEs. This result is shown to be
useful to determine the asymptotic stability of fractional-order 
systems in biology. In addition, the inequality obtained by 
Vargas-De-Le\'{o}n \cite{Vargas} for Volterra-type Lyapunov 
functions is generalized and improved.
 
On the other hand, two proposed fractional HIV models are studied to show 
the effectiveness of our method. Firstly, we demonstrated  
the global stability of the endemic equilibrium of a fractional 
SICA model based on the Lyapunov functional proposed  
by Silva and Torres \cite{Silva}, when the basic reproduction 
number is greater than one, that is, $R_0>1$. Secondly, 
we have improved the global stability of the disease-free equilibrium 
by constructing an appropriate Lyapunov functional when $R_0 \leq 1$. 
To validate these theoretical results, we carried out some numerical simulations 
for different values of the order of the fractional derivative. 
We also remarked that when the value of this order is small, 
the solution of the fractional SICA model converges rapidly to the steady-states. 
The same approach is applied to prove the global stability of any arbitrary 
equilibrium point for a fractional HIV cellular model.

Time delay is a very important element in mathematical biology 
\cite{MR3999697}. Generally, it represents the incubation time, 
the time needed for the activation of immunity or other biological 
effects \cite{MR3562914}. 
To study the global stability of delayed systems, we often 
combine Volterra-type Lyapunov functions with others depending 
on the delays. Our method can be useful in biology
to extend such functions to fractional systems with delays. 
This is under investigation
and will be addressed elsewhere. 


\section*{Acknowledgment} 

Torres is supported by the Portuguese Foundation for Science and Technology 
(FCT -- Funda\c{c}\~ao para a Ci\^encia e a Tecnologia), 
within project UIDB/04106/2020 (CIDMA). The authors are strongly grateful 
to two anonymous referees for their suggestions and invaluable comments.



\medskip

Submitted to Chaos, Solitons \& Fractals May 25, 2020; 
revised July 25, 2020; accepted for publication August 19, 2020.

\medskip


\end{document}